\documentclass[a4paper]{article}

\usepackage{amscd}
\usepackage{amsmath}
\usepackage{latexsym}
\usepackage{amsfonts}
\usepackage{amssymb}
\usepackage{graphicx,cases}
\usepackage{color,cite}
\usepackage{amsthm}
\usepackage{indentfirst}
\usepackage{bm}

\makeatletter

\@addtoreset{equation}{section}
\makeatother
\allowdisplaybreaks[4]

\newtheorem{lemma}{Lemma}[section]
\newtheorem{proposition}{Proposition}[section]
\newtheorem{theorem}{Theorem}[section]

\theoremstyle{definition}

\newtheorem{remark}{Remark}[section]

\def \p {\partial}
\def \O {\varOmega}
\def \e {\varepsilon}
\def \r {\theta}
\def \l {\lambda}
\def \d {\cdot}
\def \g {\Pi}
\def \n {\nabla}
\def \nh {\nabla_h}
\def \la {\Delta}
\def \ma {\mathcal}

\newcommand{\ie}{i.e.}
\newcommand{\eg}{e.g.}
\newcommand{\et}{\emph{et al.}}

\newcommand{\xkh}[1]{\left(#1\right)}
\newcommand{\zkh}[1]{\left[#1\right]}
\newcommand{\dkh}[1]{\left\{#1\right\}}
\newcommand{\ds}[1]{\int^t_0#1ds}
\newcommand{\dz}[1]{\int^1_{-1}#1dz}
\newcommand{\dk}[1]{\int^z_{-1}#1d\xi}
\newcommand{\mm}[1]{\int_{G}#1dxdy}
\newcommand{\oo}[1]{\int_{\O}#1dxdydz}
\newcommand{\norm}[1]{\left\lVert#1\right\rVert}
\newcommand{\ts}[1]{\int^t_0\int_{\O}#1dxdydzds}

\begin{document}
\title{Rigorous derivation of the full primitive equations by the scaled Boussinesq equations with rotation}
\author{{Xueke Pu,~{Wenli Zhou}}\\[1ex]
\normalsize{\it School of Mathematics and Information Science,}\\
\normalsize{\it Guangzhou University,~Guangzhou 510006,~China}\\
\normalsize{Email:~puxueke@gmail.com,~wywlzhou@163.com}}
\date{}

\maketitle
\begin{abstract}
The primitive equations of large-scale oceanic dynamics form a fundamental model in geophysical flows. It is well-known that the primitive equations can be formally derived by the hydrostatic approximation. On the other hand, the mathematically rigorous derivation of the primitive equations without coupling with the temperature is also known. In this paper, we generalize the above result from the mathematical point of view. More precisely, we prove that the scaled Boussinesq equations with rotation converge to the full primitive equations in a strong sense, globally in time, with the convergence rate $O(\e)$, as the aspect ratio $\e$ goes to zero.
\end{abstract}

\begin{center}
\begin{minipage}{120mm}
{\small{\bf AMS Subject Classification:~35Q35,~35Q86,~86A05,~86A10}}
\end{minipage}
\end{center}

\begin{center}
\begin{minipage}{120mm}
{\small{\bf Key Words:~Boussinesq equations;~Primitive equations;~Hydrostatic approximation;~Strong convergence}}
\end{minipage}
\end{center}

\section{Introduction}
For large-scale oceanic dynamics,~an important feature is that the vertical scale of ocean is much smaller than the horizontal scale.~From a physical point of view, this scale analysis means that we can use the hydrostatic approximation to simulate the motion of ocean in the vertical direction.~Due to this fact,~the primitive equations of oceanic dynamics can be formally derived from the Boussinesq equations with rotation (see \cite{rt1992,ct2007}).

Mathematically, there are already several results on the rigorous derivation of the primitive equations or the rigorous justification of the hydrostatic approximation.~The weak convergence from the Navier-Stokes equations to the primitive equations was firstly proved by Az\'{e}rad-Guill\'{e}n\cite{pa2001}, while the strong convergence from the Navier-Stokes equations to the primitive equations is due to the work of Li-Titi\cite{lt2019}, see also Furukawa \et\cite{kf2020} in the $L^p$ settings.~For the case where the limiting system is the primitive equations with only horizontal viscosity, we refer to the work of Li-Titi-Yuan \cite{yu2022}.~Note that the limiting system they obtained was the primitive equations without temperature.~Therefore, the aim of this paper is mathematically to derive the full primitive equations.

Let~$\O_\e=G\times(-\e,\e)$~be a $\e$-dependent domain,~where~$G=
(0,1)\times(0,1)$.~Here,~$\e=H/L$~is called the aspect ratio,~which measures the ratio of vertical scale $H$ to horizontal scale $L$ of the ocean. For large-scale ocean circulation,~the aspect ratio~$\e\sim 10^{-3}\ll 1$. Consider the anisotropic Boussinesq equations with rotation
\begin{equation}\label{eq:udn}
\begin{cases}
  \p_t u+(u \d \n)u+\n p-\r\vec{k}+f_0\vec{k} \times v=\frac{1}{Re_1}\la_h u+\frac{\e^2}{Re_2}\p_{zz}u,\\
  \p_t \r+u \d \n \r=\frac{1}{Rt_1}\la_h \r+\frac{\e^2}{Rt_2}\p_{zz}\r,\\
  \n \d u=0,
\end{cases}
\end{equation}in a thin domain $\O_\e$, where the velocity field $u=(v,w)=(v_1,v_2,w)$, pressure $p$ and temperature $\r$ are the unknowns. $f_0$ is the Coriolis parameter. The unit vector $\vec{k}=(0,0,1)$ points to the $z$-direction. The parameters $Re_1$, $Re_2$ , $Rt_1$ and $Rt_2$ are Reynolds numbers. Denote by $\nh=(\p_x,\p_y)$ the horizontal gradient operator. Then the horizontal Laplacian operator $\la_h$ is given by
\begin{equation*}
  \la_h=\nh \d \nh=\p_{xx}+\p_{yy}.
\end{equation*}

We firstly transform the $\e$-dependent domain $\O_\e$ into a fixed domain.~To this end,~we introduce some scaling transformations
\begin{gather*}
  u_{\e}=(v_\e,w_\e),~v_\e(x,y,z,t)=v(x,y,\e z,t),\\
  w_\e(x,y,z,t)=\frac{1}{\e}w(x,y,\e z,t),~p_\e(x,y,z,t)=p(x,y,\e z,t),\\
  \r_\e(x,y,z,t)=\e \r(x,y,\e z,t),
\end{gather*}for any~$(x,y,z)\in\O=:G\times(-1,1)$~and for any~$t\in(0,\infty)$.~Under these scalings,~the system~(\ref{eq:udn})~defined on~$\O_\e$~becomes the following scaled Boussinesq equations with rotation (SBE)
\begin{equation*}(SBE)
\begin{cases}
  \p_t v_\e-\frac{1}{Re_1}\la_h v_\e-\frac{1}{Re_2}\p_{zz}v_\e
  +(v_\e \d \nh)v_\e+w_\e \p_z v_\e+\nh p_\e+f_0\vec{k} \times v_\e=0,\\
  \e^2\xkh{\p_tw_\e-\frac{1}{Re_1}\la_h w_\e-\frac{1}{Re_2}\p_{zz}w_\e+v_\e \d \nh w_\e
  +w_\e \p_z w_\e}+\p_z p_\e-\r_\e=0,\\
  \p_t \r_\e-\frac{1}{Rt_1}\la_h \r_\e-\frac{1}{Rt_2}\p_{zz}\r_\e+v_\e \d \nh \r_\e+w_\e \p_z \r_\e=0,\\
  \nh \d v_\e+\p_z w_\e=0,
\end{cases}
\end{equation*}defined on the fixed domain~$\O$, subject to
\begin{gather}
  v_\e,w_\e,p_\e~\textnormal{and}~\r_\e~\textnormal{are periodic in}~x,y,z, \label{ga:are}\\
  (v_\e,w_\e,\r_\e)|_{t=0}=(v_0,w_0,\r_0), \label{ga:vew}\\
  v_\e,w_\e,p_\e~\textnormal{and}~\r_\e~\textnormal{are even,~odd,~even and odd in}~z,~\textnormal{respectively}.\label{eq:eve}
\end{gather}We point out that the condition (\ref{eq:eve}) is preserved by (SBE) provided that the initial data satisfy this condition.~Owing to this fact and (\ref{ga:are}), in this paper, we always assume that the initial data satisfy the following compatibility conditions
\begin{gather}
  v_0,w_0~\textnormal{and}~\r_0~\textnormal{are periodic in}~x,y,z,\\
  v_0,w_0~\textnormal{and}~\r_0~\textnormal{are even,~odd~and odd in}~z,~\textnormal{respectively}.\label{eq:xyz}
\end{gather}

Throughout the paper, we will use the same notation $L^p(\O)$ or $H^m(\O)$ to denote both a space itself and its finite product spaces.~For simplicity, we denote by notation $\norm{\d}_p$ the $L^p(\O)$ norm.

\begin{remark}\label{re:w0v}
The initial value $w_0$ for vertical velocity $w_\e$ is uniquely determined by the incompressible condition and compatibility condition (\ref{eq:xyz}), which can be represented as
\begin{equation}\label{eq:yxi}
  w_0(x,y,z)=-\dk{\nh \d v_0(x,y,\xi)},
\end{equation}for any~$(x,y) \in G$~and~$z \in (-1,1)$. From (\ref{eq:xyz}) and (\ref{eq:yxi}) it deduces that
\begin{equation*}
  \dz{\nh \d v_0(x,y,z)}=0.
\end{equation*}
\end{remark}

When the aspect ratio $\e$ goes to zero, (SBE) formally converges to the full primitive equations (FPE)
\begin{equation*}(FPE)
\begin{cases}
  \p_t v-\frac{1}{Re_1}\la_h v-\frac{1}{Re_2}\p_{zz}v+(v \d \nh)v+w \p_z v+\nh p+f_0\vec{k} \times v=0,\\
  \p_z p-\r=0,\\
  \p_t \r-\frac{1}{Rt_1}\la_h \r-\frac{1}{Rt_2}\p_{zz}\r+v \d \nh \r+w \p_z \r=0,\\
  \nh \d v+\p_z w=0,
\end{cases}
\end{equation*}corresponding to
\begin{gather}
  v,w,p~\textnormal{and}~\r~\textnormal{are periodic in}~x,y,z, \label{ga:pare}\\
  (v,\r)|_{t=0}=(v_0,\r_0), \label{ga:pvew}\\
  v,w,p~\textnormal{and}~\r~\textnormal{are even,~odd,~even and odd in}~z,~\textnormal{respectively}.\label{eq:peve}
\end{gather}

In consequence,~the aim of this paper is to prove the aspect ratio limit from (SBE) to (FPE).~It is crucial to point out that the global well-posedness of strong solutions to (FPE) with (\ref{ga:pare})-(\ref{eq:peve}) can be established by a similar argument in Cao-Titi\cite[Theorem 2]{ct2007}, for any initial data $(v_0,\r_0) \in H^1$.~Moreover,~we can also use the method in \cite[Theorem 2.1]{wz2021} to obtain this well-posedness result, which will play an important role in proving the aspect ratio limit from (SBE) to (FPE).~Further results on the study of the hydrostatic approximation are available in \cite{wz2021,pz2022}.

Next we want to recall some results concerning the primitive equations.~The primitive equations are considered as a fundamental model in geophysical fluid  dynamics (see, e.g., \cite{wm1986,jp1987,ds1996,am2003,gk2006}).~The global existence of weak solutions to the primitive equations with full viscosity and diffusivity was firstly given by Lions-Temam-Wang\cite{rt1992,jl1992,sw1995},~but uniqueness of weak solutions to this mathematical model is still unknown except for some special cases\cite{db2003,tt2010,ik2014,jl2017}.~Moreover,~the existence of global strong solutions to the primitive equations was proved by Kobelkov \cite{gm2006} and Kukavica-Ziane \cite{ik2007,mz2007} in the case of mixed Dirichlet and Neumann boundary conditions, see also Hieber-Kashiwabara\cite{mh2016}, Hieber \et~\cite{ah2016} and Giga \et~\cite{ym2020} in the $L^p$ settings.~In addition,~for the initial data lying in the critical solenoidal Besov space,~the unique global strong solution of primitive equations was obtained by Giga \et~\cite{mg2020}.

Subsequently, the global strong solution to the primitive equations was naturally studied in the cases of partial dissipation.~More details on these cases can be found in the work of Cao-Titi\cite{ct2012},~Fang-Han\cite{dy2020},~Li-Yuan\cite{li2022},~and Cao-Li-Titi\cite{cc2014,jl2014,es2016,cc2017,es2020}.~We remark that the inviscid primitive equations is known to be ill-posed in Sobolev spaces\cite{mr2009,dh2016},~and its smooth solutions may develop singularity in finite time\cite{cc2015,tk2015,si2021}.~However,~the local well-posedness of the inviscid primitive equations can be established under the assumption of initial data belonging to the space of analytic function (see, \eg, \cite{ku2011,te2022}).

Now we are to state the main result of this paper. Suppose that initial data $(v_0,\r_0) \in H^2(\O)$. Then from (\ref{eq:yxi}) it follows that $(v_0,w_0,\r_0) \in H^1(\O)$. By the similar argument as Lions-Temam-Wang \cite[Part IV]{rt1992}, there exists a unique local strong solution $(v_\e,w_\e,\r_\e)$ to (SBE) with (\ref{ga:are})-(\ref{eq:eve}). Denote by $T^*_\e$ the maximal existence time of the local strong solutions to (SBE). Let
\begin{gather*}
  (U_\e,\g_\e,P_\e)=(V_\e,W_\e,\g_\e,P_\e), \\
  (V_\e,W_\e,\g_\e,P_\e)=(v_\e-v,w_\e-w,\r_\e-\r,p_\e-p).
\end{gather*}Subtracting (FPE) from (SBE), then difference function $(V_\e,W_\e,\g_\e,P_\e)$ satisfies the following system
\begin{flalign}
  &\p_t V_\e-\frac{1}{Re_1}\la_h V_\e-\frac{1}{Re_2}\p_{zz}V_\e+(U_\e \d \n)V_\e+(u \d \n)V_\e\nonumber\\
  &+(U_\e \d \n)v+\nh P_\e+f_0\vec{k} \times V_\e=0, \label{fl:Ve0}\\
  &\e^2\xkh{\p_t W_\e-\frac{1}{Re_1}\la_h W_\e-\frac{1}{Re_2}\p_{zz}W_\e+U_\e \d \n W_\e
  +U_\e \d \n w+u \d \n W_\e}\nonumber\\
  &+\e^2\xkh{\p_t w+u \d \n w-\frac{1}{Re_1}\la_h w-\frac{1}{Re_2}\p_{zz}w}+\p_z P_\e-\g_\e=0, \label{fl:dnw}\\
  &\p_t \g_\e-\frac{1}{Rt_1}\la_h \g_\e-\frac{1}{Rt_2}\p_{zz}\g_\e+U_\e \d \n \g_\e+U_\e \d \n \r+u \d \n \g_\e=0, \label{fl:nge}\\
  &\nh \d V_\e+\p_z W_\e=0, \label{fl:zWe}
\end{flalign}defined on~$\O \times (0,T^*_\e)$. For this case, we have the following strong convergence theorem.

\begin{theorem}\label{th:b12}
Given a periodic function pair $(v_0,\r_0) \in H^2(\O)$ with $\dz{\nh \d v_0}=0$. Denote by $(v_\e,w_\e,\r_\e)$ and $(v,\r)$  the unique local strong solution of (SBE) subject to (\ref{ga:are})-(\ref{eq:eve}) and the unique global strong solution of (FPE) corresponding to (\ref{ga:pare})-(\ref{eq:peve}), respectively.\\
(i) For any $T>0$, there exists a small positive constant $\e(T)=\frac{2\l_0}{3\sqrt{\ma{M}_4(T)}}$ such that (SBE) has a unique strong solution $(v_\e,w_\e,\r_\e)$ on the time interval $[0,T]$, and that the following estimate holds
\begin{flalign*}
  &\sup_{0 \leq t \leq T}\xkh{\norm{(V_\e,\e W_\e,\g_\e)}^2_{H^1}}(t)+\int^T_0{\xkh{\frac{1}{Re_1}\norm{\nh V_\e}^2_{H^1}+\frac{1}{Re_2}\norm{\p_z V_\e}^2_{H^1}+\frac{\e^2}{Re_1}\norm{\nh W_\e}^2_{H^1}}}dt\\
  &\qquad+\int^T_0{\xkh{\frac{\e^2}{Re_2}\norm{\p_z W_\e}^2_{H^1}+\frac{1}{Rt_1}\norm{\nh \g_\e}^2_{H^1}
  +\frac{1}{Rt_2}\norm{\p_z \g_\e}^2_{H^1}}}dt \leq \e^2\ma{M}_5(T),
\end{flalign*}for every $\e \in (0,\e(T))$, where both $\ma{M}_4(t)$ and $\ma{M}_5(t)$ are the nonnegative continuously increasing functions that do not depend on $\e$;\\
(ii) As a consequence, we have the following strong convergences
\begin{gather*}
  w_\e \rightarrow w~in~L^{\infty}\xkh{0,T;L^2(\O)},(v_\e,\e w_\e,\r_\e) \rightarrow (v,0,\r)~in~L^{\infty}\xkh{0,T;H^1(\O)},\\
  \xkh{\frac{\nh v_\e}{\sqrt{Re_1}},\frac{\p_z v_\e}{\sqrt{Re_2}},\frac{\e\nh w_\e}{\sqrt{Re_1}}}
  \rightarrow \xkh{\frac{\nh v}{\sqrt{Re_1}},\frac{\p_z v}{\sqrt{Re_2}},0}~in~L^2\xkh{0,T;H^1(\O)},\\
  \xkh{w_\e,\frac{\e\p_z w_\e}{\sqrt{Re_2}},\frac{\nh \r_\e}{\sqrt{Rt_1}},\frac{\p_z \r_\e}{\sqrt{Rt_2}}}
  \rightarrow \xkh{w,0,\frac{\nh \r}{\sqrt{Rt_1}},\frac{\p_z \r}{\sqrt{Rt_2}}}~in~L^2\xkh{0,T;H^1(\O)},
\end{gather*}and the convergence rate is of the order $O(\e)$.
\end{theorem}
\begin{remark}
(i) The convergence results in Theorem \ref{th:b12} imply a rigorous justification of hydrostatic approximation.~In other words, (FPE) can be obtained by replacing the vertical momentum equation in (SBE) with the hydrostatic approximation
\begin{equation*}
  \p_z p-\r=0.
\end{equation*}

(ii) We point out that the convergence results here are global in time but not uniform in time compared with \cite{lt2019}.~The main reason is that the energy estimate on the strong solutions to the primitive equations with full viscosity and diffusivity is controlled by a nonnegative continuously increasing function with respect to time \cite{ct2007,wz2021}.

(iii) Roughly speaking, Proposition 4.2 in \cite{lt2019} in fact can be seen as performing the basic energy estimate for a difference function of weak solutions and strong solutions.~Therefore, this paper only focuses on the case where the initial data $(v_0,\r_0) \in H^2(\O)$, which allows us to deal with the problem in the framework of strong solutions.~In this case, we can carry out the $L^2$ estimate on the difference function directly (see Proposition \ref{po:b1T}).~The convergence of the global ``Leray-Hopf-type'' weak solutions of the Boussinesq equation to the global strong solutions of the primitive equations in different cases can be found in \cite{wz2021,pz2022}.


\end{remark}

The rest of this paper is arranged as follows. In section 2, we carry out the basic energy estimate on the system (\ref{fl:Ve0})-(\ref{fl:zWe}). In section 3, the first order energy estimate on this system is established under some smallness condition. The proof of Theorem \ref{th:b12} is presented in Section 4.

\section{The $L^2$ estimate on $(V_\e,\e W_\e,\g_\e)$}
Under the assumption of initial data $(v_0,\r_0) \in H^1(\O)$, the global well-posedness of strong solutions to (FPE) with Neumann boundary conditions was established by Cao-Titi\cite{ct2007}. Similarly, we also have the following well-posedness result for the case of periodic boundary conditions.

\begin{proposition}\label{po:lss}
Suppose that $(v_0,\r_0) \in H^1(\O)$ with $\dz{\nh \d v_0}=0$. Then the following assertions hold true:\\
(i) For any $T>0$,~there exists a unique strong solution $(v,\r)$ to (FPE) subject to (\ref{ga:pare})-(\ref{eq:peve}) such that
\begin{gather*}
  (v,\r) \in C([0,T];H^1(\O)) \cap L^2(0,T;H^2(\O)),\\
  (\p_t v,\p_t \r) \in L^2(0,T;L^2(\O));
\end{gather*}
(ii) The global strong solution $(v,\r)$ to (FPE) satisfies the following energy estimate
\begin{equation}\label{eq:a1t}
  \sup_{0 \leq s \leq t}\xkh{\norm{(v,\r)}^2_{H^1(\O)}}(s)
  +\ds{\xkh{\norm{(\n v,\n\r)}^2_{H^1(\O)}+\norm{(\p_t v,\p_t \r)}^2_2}} \leq \ma{M}_1(t),
\end{equation}for any $t \in [0,\infty)$,~where $\ma{M}_1(t)$ is a nonnegative continuously increasing function.
\end{proposition}

The proof of Proposition \ref{po:lss} is similar to that of \cite[Theorem 2]{ct2007}, see also \cite[Theorem 2.1]{wz2021}, so we omit it.~Next we present a crucial lemma (see \cite[Lemma 2.1]{jl2014}), which will be frequently used in the rest of this paper.
\begin{lemma}\label{le:phi}
The following inequalities hold
\begin{flalign*}
  \mm{&\xkh{\dz{\varphi(x,y,z)}} \xkh{\dz{\psi(x,y,z)\phi(x,y,z)}}}\\
  &\leq C\norm{\varphi}^{1/2}_2 \xkh{\norm{\varphi}^{1/2}_2+\norm{\nh\varphi}^{1/2}_2}
  \norm{\psi}^{1/2}_2 \xkh{\norm{\psi}^{1/2}_2+\norm{\nh\psi}^{1/2}_2} \norm{\phi}_2,
\end{flalign*}
\begin{flalign*}
  \mm{&\xkh{\dz{\varphi(x,y,z)}} \xkh{\dz{\psi(x,y,z)\phi(x,y,z)}}}\\
  &\leq C\norm{\psi}^{1/2}_2 \xkh{\norm{\psi}^{1/2}_2+\norm{\nh\psi}^{1/2}_2}
  \norm{\phi}^{1/2}_2 \xkh{\norm{\phi}^{1/2}_2+\norm{\nh\phi}^{1/2}_2} \norm{\varphi}_2,
\end{flalign*}for every~$\varphi,\psi,\phi$~such that the quantities on right-hand side make sense,~where~$C$~is a positive constant.
\end{lemma}

Based on Proposition \ref{po:lss}, we will perform the basic energy estimate on the system (\ref{fl:Ve0})-(\ref{fl:zWe}).
\begin{proposition}\label{po:b1T}
Assume that $(v_0,\r_0) \in H^1(\O)$ with $\dz{\nh \d v_0}=0$. Then the system~(\ref{fl:Ve0})-(\ref{fl:zWe})~has the following basic energy estimate
\begin{flalign*}
  &\sup_{0 \leq s \leq t}\xkh{\norm{(V_\e,\e W_\e,\g_\e)}^2_2}(s)+\ds{\xkh{\frac{1}{Re_1}\norm{\nh V_\e}^2_2+\frac{1}{Re_2}\norm{\p_z V_\e}^2_2}}\\
  &\qquad+\ds{\xkh{\frac{\e^2}{Re_1}\norm{\nh W_\e}^2_2+\frac{\e^2}{Re_2}\norm{\p_z W_\e}^2_2+\frac{1}{Rt_1}\norm{\nh \g_\e}^2_2+\frac{1}{Rt_2}\norm{\p_z \g_\e}^2_2}}\leq \e^2\ma{M}_2(t),
\end{flalign*}for any~$t \in [0,T^*_\e)$,~where
\begin{equation*}
  \ma{M}_2(t)=Ce^{\dkh{Ct+C(t+1)\zkh{\ma{M}_1(t)+\ma{M}^2_1(t)}}}
  \zkh{\xkh{1+t}^2+\ma{M}_1(t)+\ma{M}^2_1(t)}.
\end{equation*}Here~$C$~is a positive constant depending only on $Re_1$, $Re_2$, $Rt_1$ and $Rt_2$.
\end{proposition}
\begin{proof}[Proof.]
Taking the $L^2(\O)$ inner product of the third equation in (SBE) with $\r_\e$ and then integrating the resulting differential equation in time between $0$ to $t$ yield
\begin{equation}\label{eq:r02}
\frac{1}{2}\norm{\r_\e(t)}^2_2+\ds{\xkh{\frac{1}{Rt_1}\norm{\nh \r_\e}^2_2+\frac{1}{Rt_2}\norm{\p_z \r_\e}^2_2}}\leq\frac{1}{2}\norm{\r_0}^2_2.
\end{equation}We multiply the first two equation in (SBE) by $v_\e$ and $w_\e$ respectively and then integrate over $\O\times(0,t)$ to reach
\begin{flalign*}
  &\frac{1}{2}\xkh{\norm{v_\e}^2_2+\e^2\norm{w_\e}^2_2}(t)
  +\ds{\frac{1}{Re_1}\norm{\nh v_\e}^2_2} \\
  &\qquad+\ds{\xkh{\frac{1}{Re_2}\norm{\p_z v_\e}^2_2+\frac{\e^2}{Re_1}\norm{\nh w_\e}^2_2
  +\frac{\e^2}{Re_2}\norm{\p_z w_\e}^2_2}} \\
  &\qquad=\frac{1}{2}\xkh{\norm{v_0}^2_2+\e^2\norm{w_0}^2_2}+\ts{\r_\e w_\e}.
\end{flalign*}Owing to the H\"{o}lder inequality and Young inequality, from (\ref{eq:r02}) it deduces that
\begin{flalign}
  &\sup_{0 \leq s \leq t}\xkh{\norm{v_\e}^2_2+\e^2\norm{w_\e}^2_2}(s)
  +\ds{\frac{1}{Re_1}\norm{\nh v_\e}^2_2}\nonumber\\
  &\qquad+\ds{\xkh{\frac{1}{Re_2}\norm{\p_z v_\e}^2_2+\frac{\e^2}{Re_1}\norm{\nh w_\e}^2_2
  +\frac{\e^2}{Re_2}\norm{\p_z w_\e}^2_2}}\nonumber\\
  &\qquad\leq C\xkh{\norm{v_0}^2_2+\e^2\norm{w_0}^2_2+t\norm{\r_0}^2_2}.\label{fl:022}
\end{flalign}

Taking the $L^2(\O)$ inner product of (\ref{fl:Ve0}), (\ref{fl:dnw}) and (\ref{fl:nge}) with $V_\e$, $W_\e$ and $\g_\e$ respectively, we have
\begin{flalign}
  &\frac{1}{2}\frac{d}{dt}\xkh{\norm{V_\e}^2_2+\e^2\norm{W_\e}^2_2+\norm{\g_\e}^2_2}
  +\xkh{\frac{1}{Re_1}\norm{\nh V_\e}^2_2+\frac{1}{Re_2}\norm{\p_z V_\e}^2_2}\nonumber\\
  &\qquad\quad+\xkh{\frac{\e^2}{Re_1}\norm{\nh W_\e}^2_2+\frac{\e^2}{Re_2}\norm{\p_z W_\e}^2_2
  +\frac{1}{Rt_1}\norm{\nh \g_\e}^2_2+\frac{1}{Rt_2}\norm{\p_z \g_\e}^2_2}\nonumber\\
  &\qquad=\oo{\zkh{(U_\e \d \n)V_\e+(u \d \n)V_\e+(U_\e \d \n)v}\d (-V_\e)}\nonumber\\
  &\qquad\quad+\oo{\e^2(U_\e \d \n W_\e+U_\e \d \n w+u \d \n W_\e+u \d \n w)(-W_\e)}\nonumber\\
  &\qquad\quad+\oo{(U_\e \d \n \g_\e+U_\e \d \n \r+u \d \n \g_\e)(-\g_\e)}\nonumber\\
  &\qquad\quad+\oo{\zkh{\g_\e-\e^2\xkh{\p_t w-\frac{1}{Re_1}\la_h w-\frac{1}{Re_2}\p_{zz}w}}W_\e}\nonumber\\
  &\qquad=:I_1+I_2+I_3+I_4.\label{fl:Ver}
\end{flalign}For the first integral term $I_1$, using the H\"{o}lder inequality, Lemma \ref{le:phi} and Young inequality gives
\begin{flalign}
  I_1:&=\oo{\zkh{(U_\e \d \n)V_\e+(u \d \n)V_\e+(U_\e \d \n)v}\d (-V_\e)}\nonumber\\
  &=\oo{\zkh{(U_\e \d \n)v}\d (-V_\e)}=\oo{\zkh{(U_\e \d \n)V_\e}\d v}\nonumber\\
  &=\oo{\zkh{(V_\e \d \nh)V_\e\d v+(\nh \d V_\e)V_\e \d v}}\nonumber\\
  &\quad+\oo{(V_\e \d \p_z v)\xkh{\dk{\nh \d V_\e}}}\nonumber\\
  &\leq C\norm{\p_z v}^{1/2}_2\norm{(\p_z v,\nh \p_z v)}^{1/2}_2
  \norm{V_\e}^{1/2}_2\norm{(V_\e,\nh V_\e)}^{1/2}_2\norm{\nh V_\e}_2\nonumber\\
  &\quad+C\norm{v}_{H^1}\norm{V_\e}^{1/2}_2\xkh{\norm{V_\e}^{1/2}_2
  +\norm{\n V_\e}^{1/2}_2}\norm{\nh V_\e}_2\nonumber\\
  &\leq C\xkh{\norm{v}^2_{H^1}+\norm{v}^4_2+\norm{\n^2 v}^2_2+\norm{\n v}^4_2+\norm{\n v}^2_2\norm{\n^2 v}^2_2}\norm{V_\e}^2_2\nonumber\\
  &\quad+\frac{1}{8}\xkh{\frac{1}{Re_1}\norm{\nh V_\e}^2_2+\frac{1}{Re_2}\norm{\p_z V_\e}^2_2},\label{fl:ge2}
\end{flalign}note that the incompressible condition, integration by parts and Sobolev embedding theorem have been used. Thanks to the Lemma \ref{le:phi} and Young inequality, we obtain
\begin{flalign}
  I_2:&=\oo{\e^2(U_\e \d \n W_\e+U_\e \d \n w+u \d \n W_\e+u \d \n w)(-W_\e)}\nonumber\\
  &=\oo{\e^2(U_\e \d \n w+u \d \n w)(-W_\e)}\nonumber\\
  &=\oo{\e^2(u_\e \d \n w)(-W_\e)}=\oo{\e^2(u_\e \d \n W_\e)w}\nonumber\\
  &=\oo{\e^2\zkh{w(v_\e \d \nh W_\e)-ww_\e(\nh \d V_\e)}}\nonumber\\
  &=\oo{\e^2\xkh{-\dk{\nh \d v}}\zkh{(v_\e \d \nh W_\e)-w_\e(\nh \d V_\e)}}\nonumber\\
  &\leq C\e^2\xkh{\norm{v_\e}^4_2+\norm{v_\e}^2_2\norm{\nh v_\e}^2_2+\norm{\n v}^2_2\norm{\n^2 v}^2_2+\e^4\norm{w_\e}^4_2}\nonumber\\
  &\quad+C\e^6\norm{w_\e}^4_2\norm{\nh w_\e}^2_2+\frac{1}{8}\xkh{\frac{1}{Re_1}\norm{\nh V_\e}^2_2
  +\frac{\e^2}{Re_1}\norm{\nh W_\e}^2_2}.\label{fl:nve}
\end{flalign}A similar argument as the first integral term $I_1$ leads to
\begin{flalign}
  I_3:&=\oo{(U_\e \d \n \g_\e+U_\e \d \n \r+u \d \n \g_\e)(-\g_\e)}\nonumber\\
  &=\oo{\zkh{(V_\e \d \nh \g_\e)\r+(\nh \d V_\e)\g_\e\r+\xkh{\dk{\nh \d V_\e}}\g_\e\p_z \r}}\nonumber\\
  &\leq C\xkh{\norm{\r}_6\norm{V_\e}_3\norm{\nh \g_\e}_2+\norm{\r}_6\norm{\g_\e}_3\norm{\nh V_\e}_2}\nonumber\\
  &\quad+C\norm{\p_z \r}^{1/2}_2\norm{(\p_z \r,\nh \p_z \r)}^{1/2}_2
  \norm{\g_\e}^{1/2}_2\norm{(\g_\e,\nh \g_\e)}^{1/2}_2\norm{\nh V_\e}_2\nonumber\\
  &\leq C\xkh{\norm{\r}^2_{H^1}+\norm{\r}^4_2+\norm{\n^2 \r}^2_2+\norm{\n \r}^4_2
  +\norm{\n \r}^2_2\norm{\n^2 \r}^2_2}\xkh{\norm{V_\e}^2_2+\norm{\g_\e}^2_2}\nonumber\\
  &\quad+\frac{1}{8}\xkh{\frac{1}{Re_1}\norm{\nh V_\e}^2_2+\frac{1}{Re_2}\norm{\p_z V_\e}^2_2
  +\frac{1}{Rt_1}\norm{\nh \g_\e}^2_2+\frac{1}{Rt_2}\norm{\p_z \g_\e}^2_2}.\label{fl:e2r}
\end{flalign}Finally, it remains to deal with the last integral term $I_4$. We apply the integration by parts, H\"{o}lder inequality and Young inequality to write
\begin{flalign}
  I_4:&=\oo{\zkh{\g_\e-\e^2\xkh{\p_t w-\frac{1}{Re_1}\la_h w-\frac{1}{Re_2}\p_{zz}w}}W_\e}\nonumber\\
  &=\oo{\zkh{\e^2\xkh{-\dk{\p_t v}-\frac{1}{Re_1}\nh w}\d \nh W_\e}}\nonumber\\
  &\quad+\oo{\zkh{-\frac{\e^2}{Re_2}\p_z w (\p_z W_\e)-\g_\e\xkh{\dk{\nh \d V_\e}}}}\nonumber\\
  &\leq C\norm{\g_\e}^2_2+C\e^2\xkh{\norm{\p_t v}^2_2+\norm{\n v}^2_2+\norm{\n^2 v}^2_2}\nonumber\\
  &\quad+\frac{1}{8}\xkh{\frac{1}{Re_1}\norm{\nh V_\e}^2_2
  +\frac{\e^2}{Re_1}\norm{\nh W_\e}^2_2+\frac{\e^2}{Re_2}\norm{\p_z W_\e}^2_2}.\label{fl:nWe}
\end{flalign}Substituting (\ref{fl:ge2})-(\ref{fl:nWe}) into (\ref{fl:Ver}) and then applying the Gronwall inequality to the resulting differential equation, it follows from (\ref{eq:a1t}) and (\ref{fl:022}) that
\begin{flalign*}
  &\xkh{\norm{(V_\e,\e W_\e,\g_\e)}^2_2}(t)+\ds{\xkh{\frac{1}{Re_1}\norm{\nh V_\e}^2_2+\frac{1}{Re_2}\norm{\p_z V_\e}^2_2}}\\
  &\qquad\quad+\ds{\xkh{\frac{\e^2}{Re_1}\norm{\nh W_\e}^2_2+\frac{\e^2}{Re_2}\norm{\p_z W_\e}^2_2
  +\frac{1}{Rt_1}\norm{\nh \g_\e}^2_2+\frac{1}{Rt_2}\norm{\p_z \g_\e}^2_2}}\\
  &\qquad\leq\exp\bigg\{C\ds{\xkh{1+\norm{v}^2_{H^1}+\norm{v}^4_2+\norm{\n^2 v}^2_2+\norm{\n v}^4_2+\norm{\n v}^2_2\norm{\n^2 v}^2_2}}\\
  &\qquad\quad+\ds{\xkh{\norm{\r}^2_{H^1}+\norm{\r}^4_2+\norm{\n^2 \r}^2_2+\norm{\n \r}^4_2+\norm{\n \r}^2_2\norm{\n^2 \r}^2_2}}\bigg\}\\
  &\qquad\quad\times\bigg\{C\e^2\ds{\xkh{\e^4\norm{w_\e}^4_2\norm{\nh w_\e}^2_2+\norm{\p_t v}^2_2+\norm{\n v}^2_2+\norm{\n^2 v}^2_2}}\\
  &\qquad\quad+\ds{\xkh{\norm{v_\e}^4_2+\norm{v_\e}^2_2\norm{\nh v_\e}^2_2+\norm{\n v}^2_2\norm{\n^2 v}^2_2+\e^4\norm{w_\e}^4_2}}\bigg\}\\
  &\qquad\leq C\e^2e^{\dkh{Ct+C(t+1)\zkh{\ma{M}_1(t)+\ma{M}^2_1(t)}}}
  \zkh{\xkh{1+t}^2+\ma{M}_1(t)+\ma{M}^2_1(t)},
\end{flalign*}completing the proof.
\end{proof}

\section{The $L^2$ estimate on $(\n V_\e,\e\n W_\e,\n\g_\e)$}
For the purpose of establishing the first order energy estimate on the system (\ref{fl:Ve0})-(\ref{fl:zWe}), we need to carry out the second order energy estimate on (FPE) corresponding to (\ref{ga:pare})-(\ref{eq:peve}). To this end, we firstly employ (\ref{eq:peve}) to rewrite (FPE). Integrating the second equation in (FPE) with respect to $z$ yields
\begin{equation}\label{eq:pxy}
  p(x,y,z,t)=p_\nu(x,y,t)+\dk{\r(x,y,\xi,t)},
\end{equation}where~$p_\nu(x,y,t)$~represents unknown surface pressure as~$z=0$.~According to (\ref{eq:pxy}) and the incompressible condition,~we can reformulate (FPE) as
\begin{flalign}
  &\p_t v-\frac{1}{Re_1}\la_h v-\frac{1}{Re_2}\p_{zz}v+(v \d \nh)v-\xkh{\dk{\nh \d v(x,y,\xi,t)}}\p_z v\nonumber\\
  &+\nh p_\nu(x,y,t)+\dk{\nh\r(x,y,\xi,t)}+f_0 \vec{k} \times v=0, \label{fl:dkr}\\
  &\p_t \r-\frac{1}{Rt_1}\la_h \r-\frac{1}{Rt_2}\p_{zz}\r+v \d \nh \r-\xkh{\dk{\nh \d v(x,y,\xi,t)}}\p_z \r=0, \label{fl:lar}
\end{flalign}subject to
\begin{gather*}
  v~\textnormal{and}~\r~\textnormal{are periodic in}~x,y,z, \\
  (v,\r)|_{t=0}=(v_0,\r_0),\\
  v~\textnormal{and}~\r~\textnormal{are even and odd in}~z,~\textnormal{respectively}.
\end{gather*}

\begin{proposition}\label{po:a2t}
Suppose that $(v_0,\r_0) \in H^2(\O)$ with $\dz{\nh \d v_0}=0$.~Then (FPE) has the following second order energy estimate
\begin{equation*}
  \sup_{0 \leq s \leq t}\xkh{\norm{\la v}^2_2+\norm{\la \r}^2_2}(s)
  +\ds{\xkh{\norm{\n \la v}^2_2+\norm{\n \p_t v}^2_2+\norm{\n \la \r}^2_2+\norm{\n \p_t \r}^2_2}} \leq \ma{M}_3(t),
\end{equation*}for any~$t \in [0,\infty)$,~where
\begin{equation*}
  \ma{M}_3(t)=Ce^{C(t+1)\zkh{1+\ma{M}_1(t)+\ma{M}^2_1(t)}}\zkh{\norm{v_0}^2_{H^2}+\norm{\r_0}^2_{H^2}+\ma{M}_1(t)}.
\end{equation*}Here $C$ is a positive constant depending only on $Re_1$, $Re_2$, $Rt_1$, $Rt_2$ and $|f_0|$.
\end{proposition}
\begin{proof}[Proof.]
Applying the gradient operator $\n$ to the equation (\ref{fl:dkr}), then taking the dot product of the resulting equation with $\n\xkh{\p_t v-\la v}$, and finally integrating over $\O$, we have
\begin{flalign}
  \frac{1}{2}\frac{d}{dt}&\xkh{\norm{\la v}^2_2+\frac{1}{Re_1}\norm{\n\nh v}^2_2+\frac{1}{Re_2}\norm{\n\p_z v}^2_2}\nonumber\\
  &\quad+\xkh{\frac{1}{Re_1}\norm{\nh\la v}^2_2+\frac{1}{Re_2}\norm{\p_z\la v}^2_2+\norm{\n \p_t v}^2_2}\nonumber\\
  &=\oo{\n\zkh{\xkh{v \d \nh}v-\xkh{\dk{\nh \d v}} \p_z v}:\n\xkh{\la v-\p_t v}}\nonumber\\
  &\quad+\oo{\zkh{\n\xkh{\dk{\nh \r}}:\n\xkh{\la v-\p_t v}-\n(f_0\vec{k}\times v):\n\p_t v}}. \label{fl:lav}
\end{flalign}note that we have used the following fact that
\begin{gather*}
  \oo{\n\nh p_\nu(x,y,t):\n\xkh{\p_t v-\la v}}=0,\\
  \oo{\n(f_0\vec{k}\times v):\n\la v}=0.
\end{gather*}In order to estimate the first integral term on the right-hand side of (\ref{fl:lav}), the gradient operator $\n$ will be divided into two parts, $\nh$ and $\p_z$. Then we use the H\"{o}lder inequality, Sobolev embedding theorem, Lemma \ref{le:phi} and Young inequality to obtain
\begin{flalign}
  &\oo{\n\zkh{\xkh{v \d \nh}v-\xkh{\dk{\nh \d v}} \p_z v}:\n\xkh{\la v-\p_t v}}\nonumber\\
  &\qquad=\oo{\zkh{\xkh{\p_j v \d \nh}v-\xkh{\dk{\nh \d \p_j v}}\p_z v}\d\xkh{\p_j \la v-\p_j \p_t v}}\nonumber\\
  &\qquad\quad+\oo{\zkh{\xkh{v \d \nh}\p_j v-\xkh{\dk{\nh \d v}}\p_j\p_z v}\d\xkh{\p_j \la v-\p_j \p_t v}}\nonumber\\
  &\qquad\quad+\oo{\zkh{\xkh{\p_z v \d \nh}v-\xkh{\nh \d v}\p_z v}\d\xkh{\p_z \la v-\p_z \p_t v}}\nonumber\\
  &\qquad\quad+\oo{\zkh{\xkh{v \d \nh}\p_z v-\xkh{\dk{\nh \d v}}\p_{zz}v}\d\xkh{\p_z \la v-\p_z \p_t v}}\nonumber\\
  &\qquad\leq C\xkh{\norm{v}^2_{H^1}+\norm{\n^2 v}^2_2+\norm{v}^4_{H^1}+\norm{\n v}^2_2\norm{\n^2 v}^2_2}
  \xkh{\frac{1}{Re_1}\norm{\n\nh v}^2_2+\frac{1}{Re_2}\norm{\n\p_z v}^2_2}\nonumber\\
  &\qquad\quad+\frac{1}{6}\xkh{\frac{1}{Re_1}\norm{\nh\la v}^2_2+\frac{1}{Re_2}\norm{\p_z\la v}^2_2+\norm{\n \p_t v}^2_2},\label{fl:pjp}
\end{flalign}where $\p_j \in \{\p_x,\p_y\}$.~The H\"{o}lder inequality and Young inequality gives
\begin{flalign}
  &\oo{\zkh{\n\xkh{\dk{\nh \r}}:\n\xkh{\la v-\p_t v}-\n(f_0\vec{k}\times v):\n\p_t v}}\nonumber\\
  &\qquad=\oo{\zkh{\dk{\p_j \nh \r}} \d \xkh{\p_j \la v-\p_j \p_t v}}\nonumber\\
  &\qquad\quad+\oo{\zkh{\nh \r \d \xkh{\p_z \la v-\p_z \p_t v}-\n(f_0\vec{k}\times v):\n\p_t v}}\nonumber\\
  &\qquad\leq C\xkh{\norm{\n v}^2_2+\norm{\n \r}^2_2}+\frac{C}{Rt_1}\norm{\n\nh \r}^2_2\nonumber\\
  &\qquad\quad+\frac{1}{6}\xkh{\frac{1}{Re_1}\norm{\nh\la v}^2_2+\frac{1}{Re_2}\norm{\p_z\la v}^2_2+\norm{\n \p_t v}^2_2}.\label{fl:fok}
\end{flalign}By substituting (\ref{fl:pjp}) and (\ref{fl:fok}) into (\ref{fl:lav}), we reach
\begin{flalign}
  \frac{1}{2}\frac{d}{dt}&\xkh{\norm{\la v}^2_2+\frac{1}{Re_1}\norm{\n\nh v}^2_2+\frac{1}{Re_2}\norm{\n\p_z v}^2_2}\nonumber\\
  &\quad+\xkh{\frac{1}{Re_1}\norm{\nh\la v}^2_2+\frac{1}{Re_2}\norm{\p_z\la v}^2_2+\norm{\n \p_t v}^2_2}\nonumber\\
  &\leq\frac{C}{Rt_1}\norm{\n\nh \r}^2_2+\frac{1}{3}\xkh{\frac{1}{Re_1}\norm{\nh\la v}^2_2+\frac{1}{Re_2}\norm{\p_z\la v}^2_2+\norm{\n \p_t v}^2_2}\nonumber\\
  &\quad+C\xkh{\norm{v}^2_{H^1}+\norm{\n^2 v}^2_2+\norm{v}^4_{H^1}+\norm{\n v}^2_2\norm{\n^2 v}^2_2}\nonumber\\
  &\quad\times\xkh{\frac{1}{Re_1}\norm{\n\nh v}^2_2+\frac{1}{Re_2}\norm{\n\p_z v}^2_2}+C\xkh{\norm{\n v}^2_2+\norm{\n \r}^2_2}.\label{fl:Rev}
\end{flalign}

Similarly, applying the gradient operator $\n$ to the equation (\ref{fl:lar}), multiplying the resulting equation by $\n\xkh{\p_t \r-\la \r}$, and integrating over $\O$, from the H\"{o}lder inequality, Sobolev embedding theorem, Lemma \ref{le:phi} and Young inequality it deduces that
\begin{flalign}
  \frac{1}{2}\frac{d}{dt}&\xkh{\norm{\la \r}^2_2+\frac{1}{Rt_1}\norm{\n\nh \r}^2_2+\frac{1}{Rt_2}\norm{\n\p_z \r}^2_2}\nonumber\\
  &\quad+\xkh{\frac{1}{Rt_1}\norm{\nh\la \r}^2_2+\frac{1}{Rt_2}\norm{\p_z\la \r}^2_2+\norm{\n \p_t \r}^2_2}\nonumber\\
  &=\oo{\n\zkh{v \d \nh\r-\xkh{\dk{\nh \d v}} \p_z \r}\d\n\xkh{\la \r-\p_t \r}}\nonumber\\
  &\leq C\Big[\xkh{\norm{v}^2_{H^1}+\norm{\n^2 v}^2_2+\norm{v}^4_{H^1}
  +\norm{\n v}^2_2\norm{\n^2 v}^2_2+\norm{\n\r}^2_{H^1}+\norm{\n\r}^4_2}\nonumber\\
  &\quad+\norm{\n \r}^2_2\norm{\n^2 \r}^2_2\Big]\xkh{\frac{1}{Re_1}\norm{\n\nh v}^2_2
  +\frac{1}{Rt_1}\norm{\n\nh \r}^2_2+\frac{1}{Rt_2}\norm{\n\p_z \r}^2_2}\nonumber\\
  &\quad+\frac{1}{6}\xkh{\frac{1}{Re_1}\norm{\nh\la v}^2_2+\frac{1}{Rt_1}\norm{\nh\la \r}^2_2
  +\frac{1}{Rt_2}\norm{\p_z\la \r}^2_2+\norm{\n \p_t \r}^2_2}.\label{fl:Rtr}
\end{flalign}Adding (\ref{fl:Rev}) and (\ref{fl:Rtr}), and using the Gronwall inequality, we obtain
\begin{flalign*}
  &\xkh{\norm{\la v}^2_2+\frac{1}{Re_1}\norm{\n\nh v}^2_2+\frac{1}{Re_2}\norm{\n\p_z v}^2_2+\norm{\la \r}^2_2+\frac{1}{Rt_1}\norm{\n\nh \r}^2_2}(t)\nonumber\\
  &\qquad\quad+\frac{1}{Rt_2}\norm{\n\p_z \r}^2_2(t)+\ds{\xkh{\frac{1}{Re_1}\norm{\nh\la v}^2_2+\frac{1}{Re_2}\norm{\p_z\la v}^2_2+\norm{\n \p_t v}^2_2}}\nonumber\\
  &\qquad\quad+\ds{\xkh{\frac{1}{Rt_1}\norm{\nh\la \r}^2_2+\frac{1}{Rt_2}\norm{\p_z\la \r}^2_2+\norm{\n \p_t \r}^2_2}}\nonumber\\
  &\qquad\leq\exp\bigg\{C\ds{\xkh{\norm{v}^2_{H^1}+\norm{\n^2 v}^2_2+\norm{v}^4_{H^1}+\norm{\n v}^2_2\norm{\n^2 v}^2_2}}\nonumber\\
  &\qquad\quad+C\ds{\xkh{1+\norm{\n\r}^2_{H^1}+\norm{\n\r}^4_2+\norm{\n \r}^2_2\norm{\n^2 \r}^2_2}}\bigg\}\nonumber\\
  &\qquad\quad\times\bigg\{\norm{\la v_0}^2_2+\frac{1}{Re_1}\norm{\n\nh v_0}^2_2+\frac{1}{Re_2}\norm{\n\p_z v_0}^2_2+\frac{1}{Rt_1}\norm{\n\nh \r_0}^2_2\nonumber\\
  &\qquad\quad+\norm{\la \r_0}^2_2+\frac{1}{Rt_2}\norm{\n\p_z \r_0}^2_2+C\ds{\xkh{\norm{\n v}^2_2+\norm{\n \r}^2_2}}\bigg\},
\end{flalign*}which yields the result in Proposition \ref{po:a2t} by (\ref{eq:a1t}). The proof is thus completed.
\end{proof}

By means of Proposition \ref{po:lss}, \ref{po:b1T} and \ref{po:a2t}, we can perform the first order energy estimate on the system (\ref{fl:Ve0})-(\ref{fl:zWe}) under some smallness condition.

\begin{proposition}\label{po:b2T}
Assume that $(v_0,\r_0) \in H^2(\O)$ with $\dz{\nh \d v_0}=0$.~Then there exists a small positive constant $\l_0$ such that the system (\ref{fl:Ve0})-(\ref{fl:zWe}) has the following first order energy estimate
\begin{flalign*}
  &\sup_{0 \leq s \leq t}\xkh{\norm{\n(V_\e,\e W_\e,\g_\e)}^2_2}(s)+\ds{\xkh{\frac{1}{Re_1}\norm{\n\nh V_\e}^2_2
  +\frac{1}{Re_2}\norm{\n\p_z V_\e}^2_2+\frac{\e^2}{Re_1}\norm{\n\nh W_\e}^2_2}}\\
  &\qquad+\ds{\xkh{\frac{\e^2}{Re_2}\norm{\n\p_z W_\e}^2_2+\frac{1}{Rt_1}\norm{\n\nh \g_\e}^2_2
  +\frac{1}{Rt_2}\norm{\n\p_z \g_\e}^2_2}}\leq \e^2\ma{M}_4(t),
\end{flalign*}for any~$t \in [0,T^*_\e)$, provided that
\begin{equation*}
  \sup_{0 \leq s \leq t}\xkh{\norm{\n (V_\e,\g_\e)}^2_2+\e^2\norm{\n W_\e}^2_2}(s) \leq \l^2_0,
\end{equation*}where
\begin{flalign*}
  \ma{M}_4(t)&=C\exp\dkh{C(t+1)\zkh{1+\ma{M}_1(t)+\ma{M}_2(t)+\ma{M}_3(t)+\ma{M}^2_1(t)+\ma{M}^2_2(t)}}\\
  &\quad\times\dkh{\ma{M}_1(t)+\ma{M}_3(t)+(t+1)\zkh{\ma{M}^2_1(t)+\ma{M}^2_2(t)+\ma{M}^2_3(t)}}.
\end{flalign*}Here $C$ is a positive constant depending only on $Re_1$, $Re_2$, $Rt_1$ and $Rt_2$.
\end{proposition}
\begin{proof}[Proof.]
Taking the $L^2(\O)$ inner product of (\ref{fl:Ve0}), (\ref{fl:dnw}) and (\ref{fl:nge}) with $-\la V_\e$, $-\la W_\e$ and $-\la \g_\e$, respectively, then from integration by parts it follows that
\begin{flalign*}
  \frac{1}{2}\frac{d}{dt}&\xkh{\norm{\n (V_\e,\g_\e)}^2_2+\e^2\norm{\n W_\e}^2_2}
  +\xkh{\frac{1}{Re_1}\norm{\n\nh V_\e}^2_2+\frac{1}{Re_2}\norm{\n\p_z V_\e}^2_2}\\
  &\quad+\xkh{\frac{\e^2}{Re_1}\norm{\n\nh W_\e}^2_2+\frac{\e^2}{Re_2}\norm{\n\p_z W_\e}^2_2
  +\frac{1}{Rt_1}\norm{\n\nh \g_\e}^2_2+\frac{1}{Rt_2}\norm{\n\p_z \g_\e}^2_2}\\
  &=\e^2\oo{\xkh{U_\e \d \n W_\e+U_\e \d \n w+u \d \n W_\e+u \d \n w}\la W_\e}\\
  &\quad+\oo{\zkh{\e^2\xkh{\p_t w-\frac{1}{Re_1}\la_h w-\frac{1}{Re_2}\p_{zz}w}-\g_\e}\la W_\e}\\
  &\quad+\oo{(U_\e \d \n \g_\e+U_\e \d \n \r+u \d \n \g_\e)\la \g_\e}\\
  &\quad+\oo{\zkh{(U_\e \d \n)V_\e+(u \d \n)V_\e+(U_\e \d \n)v} \d \la V_\e}\\
  &=:R_1+R_2+R_3+R_4,
\end{flalign*}where we have used the following fact that
\begin{equation*}
  \oo{f_0(\vec{k}\times V_\e) \d \la V_\e}=0.
\end{equation*}Noting that the fact $|v| \leq \frac{1}{2}\dz{|v|}+\dz{|\p_{z}v|}$, we use the incompressible condition, Lemma \ref{le:phi} and Young inequality to obtain
\begin{flalign}
  R_1:&=\e^2\oo{\xkh{U_\e \d \n W_\e+U_\e \d \n w+u \d \n W_\e+u \d \n w}\la W_\e}\nonumber\\
  &=\e^2\oo{\zkh{V_\e \d \nh W_\e-(\p_z W_\e)\dk{\nh \d V_\e}}\la W_\e}\nonumber\\
  &\quad+\e^2\oo{\zkh{(\nh \d v)\dk{\nh \d V_\e}-V_\e \d \dk{\nh(\nh \d v)}}\la W_\e}\nonumber\\
  &\quad+\e^2\oo{\zkh{v \d \nh W_\e+(\nh \d V_\e)\dk{\nh \d v}}\la W_\e}\nonumber\\
  &\quad+\e^2\oo{\zkh{(\nh \d v)\dk{(\nh \d v)}-v \d \dk{\nh(\nh \d v)}}\la W_\e}\nonumber\\
  &\leq\e^2\mm{\xkh{\dz{(|V_\e|+|\p_z V_\e|)}}\xkh{\dz{|\nh W_\e||\la W_\e|}}}\nonumber\\
  &\quad+\e^2\mm{\xkh{\dz{|\nh V_\e|}}\xkh{\dz{|\p_z W_\e||\la W_\e|}}}\nonumber\\
  &\quad+\e^2\mm{\xkh{\dz{|\nh V_\e|}}\xkh{\dz{|\nh v||\la W_\e|}}}\nonumber\\
  &\quad+\e^2\mm{\xkh{\dz{|\nh^2 v|}}\xkh{\dz{|V_\e||\la W_\e|}}}\nonumber\\
  &\quad+\e^2\mm{\xkh{\dz{(|v|+|\p_z v|)}}\xkh{\dz{|\nh W_\e||\la W_\e|}}}\nonumber\\
  &\quad+\e^2\mm{\xkh{\dz{|\nh v|}}\xkh{\dz{|\nh V_\e||\la W_\e|}}}\nonumber\\
  &\quad+\e^2\mm{\xkh{\dz{|\nh v|}}\xkh{\dz{|\nh v||\la W_\e|}}}\nonumber\\
  &\quad+\e^2\mm{\xkh{\dz{|\nh^2 v|}}\xkh{\dz{|v||\la W_\e|}}}\nonumber\\
  &\leq C\bigg\{\xkh{\frac{1}{Re_1}\norm{\n\nh V_\e}^2_2+\frac{1}{Re_2}\norm{\n\p_z V_\e}^2_2
  +\frac{\e^2}{Re_1}\norm{\n\nh W_\e}^2_2+\frac{\e^2}{Re_2}\norm{\n\p_z W_\e}^2_2}\nonumber\\
  &\quad+\zkh{(1+\e^2)\norm{V_\e}^2_2+\norm{\nh V_\e}^2_2+\norm{\p_z V_\e}^2_2
  +\norm{V_\e}^4_2+\norm{V_\e}^2_2\norm{\nh V_\e}^2_2}\nonumber\\
  &\quad+\zkh{\norm{v}^2_{H^1}+\norm{\n^2 v}^2_2+\norm{v}^4_2
  +(1+\e^4)\norm{v}^2_{H^1}\norm{\n v}^2_{H^1}}\bigg\}\norm{\n (V_\e,\e W_\e)}^2_2\nonumber\\
  &\quad+C\e^2\xkh{\norm{v}^4_2+\norm{v}^2_{H^1}\norm{\n v}^2_{H^1}
  +\norm{\n^2 v}^4_2+\norm{\n^2 v}^2_2\norm{\n\la v}^2_2+\norm{V_\e}^4_2}\nonumber\\
  &\quad+\frac{1}{48}\xkh{\frac{1}{Re_1}\norm{\n\nh V_\e}^2_2+\frac{\e^2}{Re_1}\norm{\n\nh W_\e}^2_2
  +\frac{\e^2}{Re_2}\norm{\n\p_z W_\e}^2_2}.\label{fl:laV}
\end{flalign}Due to the the incompressible condition, H\"{o}lder inequality and Young inequality, we reach
\begin{flalign}
  R_2:&=\oo{\zkh{\e^2\xkh{\p_t w-\frac{1}{Re_1}\la_h w-\frac{1}{Re_2}\p_{zz}w}-\g_\e}\la W_\e}\nonumber\\
  &=\e^2\oo{\xkh{-\dk{\nh\d\p_t v}+\frac{1}{Re_1}\dk{\la_h(\nh \d v)}+\frac{1}{Re_2}\nh \d \p_z v}\la W_\e}\nonumber\\
  &\quad+\oo{\zkh{-\nh \g_\e \d \xkh{\dk{\nh(\nh \d V_\e)}}-\p_z\g_\e\xkh{\nh \d V_\e}}}\nonumber\\
  &\leq C\e^2\xkh{\norm{\n \p_t v}^2_2+\norm{\n^2 v}^2_2+\norm{\n \la v}^2_2}+C\xkh{\norm{\n V_\e}^2_2+\norm{\n \g_\e}^2_2}\nonumber\\
  &\quad+\frac{1}{48}\xkh{\frac{1}{Re_1}\norm{\n\nh V_\e}^2_2
  +\frac{\e^2}{Re_1}\norm{\n\nh W_\e}^2_2+\frac{\e^2}{Re_2}\norm{\n\p_z W_\e}^2_2}.\label{fl:law}
\end{flalign}By the similar method as the integral term $R_1$, the integral term $R_3$ can be bounded as
\begin{flalign}
  R_3:&=\oo{(U_\e \d \n \g_\e+U_\e \d \n \r+u \d \n \g_\e)\la \g_\e}\nonumber\\
  &=\oo{\zkh{V_\e \d \nh \g_\e-(\p_z \g_\e)\dk{\nh \d V_\e}+V_\e \d \nh \r}\la\g_\e}\nonumber\\
  &\quad+\oo{\zkh{v \d \nh\g_\e-(\p_z \r)\dk{\nh \d V_\e}-(\p_z \g_\e)\dk{\nh \d v}}\la\g_\e}\nonumber\\
  &\leq C\bigg\{\xkh{\frac{1}{Re_1}\norm{\n\nh V_\e}^2_2+\frac{1}{Re_2}\norm{\n\p_z V_\e}^2_2
  +\frac{1}{Rt_1}\norm{\n\nh \g_\e}^2_2+\frac{1}{Rt_2}\norm{\n\p_z \g_\e}^2_2}\nonumber\\
  &\quad+\xkh{\norm{v}^4_2+\norm{\n^2 v}^2_2+\norm{v}^2_{H^1}\norm{\n v}^2_{H^1}
  +\norm{\n\r}^2_{H^1}+\norm{\r}^2_{H^1}\norm{\n \r}^2_{H^1}}\nonumber\\
  &\quad+\xkh{\norm{V_\e}^4_2+\norm{\nh V_\e}^2_2+\norm{\nh \g_\e}^2_2+\norm{V_\e}^2_2\norm{\nh V_\e}^2_2}\bigg\}
  \xkh{\norm{\n V_\e}^2_2+\norm{\n \g_\e}^2_2}\nonumber\\
  &\quad+C\norm{V_\e}^2_2\norm{\n^2\r}^2_2+\frac{1}{48}\xkh{\frac{1}{Re_1}\norm{\n\nh V_\e}^2_2
  +\frac{1}{Re_2}\norm{\n\p_z V_\e}^2_2}\nonumber\\
  &\quad+\frac{1}{48}\xkh{\frac{1}{Rt_1}\norm{\n\nh \g_\e}^2_2+\frac{1}{Rt_2}\norm{\n\p_z \g_\e}^2_2}.\label{fl:221}
\end{flalign}Similarly, we have
\begin{flalign}
  R_4:&=\oo{\zkh{(U_\e \d \n)V_\e+(u \d \n)V_\e+(U_\e \d \n)v} \d \la V_\e}\nonumber\\
  &=\oo{\zkh{(V_\e \d \nh)V_\e-(\p_z V_\e)\dk{\nh \d V_\e}} \d \la V_\e}\nonumber\\
  &\quad+\oo{\zkh{(V_\e \d \nh)v-(\p_z v)\dk{\nh \d V_\e}} \d \la V_\e}\nonumber\\
  &\quad+\oo{\zkh{(v \d \nh)V_\e-(\p_z V_\e)\dk{\nh \d v}} \d \la V_\e}\nonumber\\
  &\leq C\xkh{\norm{v}^2_{H^1}\norm{\n v}^2_{H^1}+\norm{V_\e}^4_2
  +\norm{\p_z V_\e}^2_2+\norm{V_\e}^2_2\norm{\nh V_\e}^2_2}\norm{\n V_\e}^2_2\nonumber\\
  &\quad+C\xkh{\norm{v}^4_2+\norm{\n v}^2_{H^1}+\frac{1}{Re_1}\norm{\n\nh V_\e}^2_2
  +\frac{1}{Re_2}\norm{\n\p_z V_\e}^2_2}\norm{\n V_\e}^2_2\nonumber\\
  &\quad+C\norm{\n^2 v}^2_2\norm{V_\e}^2_2+\frac{1}{48}\xkh{\frac{1}{Re_1}\norm{\n\nh V_\e}^2_2
  +\frac{1}{Re_2}\norm{\n\p_z V_\e}^2_2}.\label{fl:nla}
\end{flalign}Adding (\ref{fl:laV}), (\ref{fl:law}), (\ref{fl:221}) and (\ref{fl:nla}) yields
\begin{flalign}
  \frac{1}{2}\frac{d}{dt}&\xkh{\norm{\n (V_\e,\g_\e)}^2_2+\e^2\norm{\n W_\e}^2_2}
  +\frac{11}{12}\xkh{\frac{1}{Re_1}\norm{\n\nh V_\e}^2_2+\frac{1}{Re_2}\norm{\n\p_z V_\e}^2_2}\nonumber\\
  &\quad+\frac{11}{12}\xkh{\frac{\e^2}{Re_1}\norm{\n\nh W_\e}^2_2+\frac{\e^2}{Re_2}\norm{\n\p_z W_\e}^2_2
  +\frac{1}{Rt_1}\norm{\n\nh \g_\e}^2_2+\frac{1}{Rt_2}\norm{\n\p_z \g_\e}^2_2}\nonumber\\
  &\leq C_\sigma\bigg\{\xkh{\frac{1}{Re_1}\norm{\n\nh V_\e}^2_2+\frac{1}{Re_2}\norm{\n\p_z V_\e}^2_2
  +\frac{\e^2}{Re_1}\norm{\n\nh W_\e}^2_2+\frac{\e^2}{Re_2}\norm{\n\p_z W_\e}^2_2}\nonumber\\
  &\quad+\zkh{\frac{1}{Rt_1}\norm{\n\nh \g_\e}^2_2+\frac{1}{Rt_2}\norm{\n\p_z \g_\e}^2_2
  +\xkh{\norm{\nh \g_\e}^2_2+\norm{\n\r}^2_{H^1}+\norm{\r}^2_{H^1}\norm{\n \r}^2_{H^1}}}\nonumber\\
  &\quad+\zkh{1+(1+\e^2)\norm{V_\e}^2_2+\norm{\nh V_\e}^2_2+\norm{\p_z V_\e}^2_2
  +\norm{V_\e}^4_2+\norm{V_\e}^2_2\norm{\nh V_\e}^2_2}\nonumber\\
  &\quad+\zkh{\norm{v}^2_{H^1}+\norm{\n^2 v}^2_2+\norm{v}^4_2
  +(1+\e^4)\norm{v}^2_{H^1}\norm{\n v}^2_{H^1}}\bigg\}\norm{\n (V_\e,\e W_\e,\g_\e)}^2_2\nonumber\\
  &\quad+C_\sigma\e^2\xkh{\norm{v}^4_2+\norm{v}^2_{H^1}\norm{\n v}^2_{H^1}
  +\norm{\n^2 v}^4_2+\norm{\n^2 v}^2_2\norm{\n\la v}^2_2+\norm{V_\e}^4_2}\nonumber\\
  &\quad+C_\sigma\e^2\xkh{\norm{\n \p_t v}^2_2+\norm{\n^2 v}^2_2+\norm{\n \la v}^2_2}
  +C_\sigma\xkh{\norm{\n^2 v}^2_2+\norm{\n^2 \r}^2_2}\norm{V_\e}^2_2.\label{fl:VWr}
\end{flalign}Setting $\l_0=\sqrt{\frac{5}{12C_\sigma}}$, and then using the following smallness condition
\begin{equation*}
  \sup_{0 \leq s \leq t}\xkh{\norm{\n (V_\e,\g_\e)}^2_2+\e^2\norm{\n W_\e}^2_2}(s) \leq \l^2_0,
\end{equation*}it deduces from (\ref{fl:VWr}) that
\begin{flalign*}
  \frac{d}{dt}&\xkh{\norm{\n (V_\e,\g_\e)}^2_2+\e^2\norm{\n W_\e}^2_2}
  +\xkh{\frac{1}{Re_1}\norm{\n\nh V_\e}^2_2+\frac{1}{Re_2}\norm{\n\p_z V_\e}^2_2}\\
  &\quad+\xkh{\frac{\e^2}{Re_1}\norm{\n\nh W_\e}^2_2+\frac{\e^2}{Re_2}\norm{\n\p_z W_\e}^2_2
  +\frac{1}{Rt_1}\norm{\n\nh \g_\e}^2_2+\frac{1}{Rt_2}\norm{\n\p_z \g_\e}^2_2}\\
  &\leq C_\sigma\bigg\{\zkh{1+\norm{\n\r}^2_{H^1}+\norm{\r}^2_{H^1}\norm{\n \r}^2_{H^1}
  +\norm{\nh \g_\e}^2_2+(1+\e^2)\norm{V_\e}^2_2}\nonumber\\
  &\quad+\xkh{\norm{\nh V_\e}^2_2+\norm{\p_z V_\e}^2_2
  +\norm{V_\e}^4_2+\norm{V_\e}^2_2\norm{\nh V_\e}^2_2+\norm{v}^2_{H^1}+\norm{\n^2 v}^2_2}\nonumber\\
  &\quad+\zkh{\norm{v}^4_2+(1+\e^4)\norm{v}^2_{H^1}\norm{\n v}^2_{H^1}}\bigg\}
  \norm{\n (V_\e,\e W_\e,\g_\e)}^2_2\nonumber\\
  &\quad+C_\sigma\e^2\xkh{\norm{v}^4_2+\norm{v}^2_{H^1}\norm{\n v}^2_{H^1}
  +\norm{\n^2 v}^4_2+\norm{\n^2 v}^2_2\norm{\n\la v}^2_2+\norm{V_\e}^4_2}\nonumber\\
  &\quad+C_\sigma\e^2\xkh{\norm{\n \p_t v}^2_2+\norm{\n^2 v}^2_2+\norm{\n \la v}^2_2}
  +C_\sigma\xkh{\norm{\n^2 v}^2_2+\norm{\n^2 \r}^2_2}\norm{V_\e}^2_2.
\end{flalign*}By virtue of the Gronwall inequality, from Proposition \ref{po:lss}, \ref{po:b1T} and \ref{po:a2t} it follows that
\begin{flalign*}
  &\xkh{\norm{\n (V_\e,\g_\e)}^2_2+\e^2\norm{\n W_\e}^2_2}(t)
  +\ds{\xkh{\frac{1}{Re_1}\norm{\n\nh V_\e}^2_2+\frac{1}{Re_2}\norm{\n\p_z V_\e}^2_2}}\\
  &\quad\quad+\ds{\xkh{\frac{\e^2}{Re_1}\norm{\n\nh W_\e}^2_2+\frac{\e^2}{Re_2}\norm{\n\p_z W_\e}^2_2
  +\frac{1}{Rt_1}\norm{\n\nh \g_\e}^2_2+\frac{1}{Rt_2}\norm{\n\p_z \g_\e}^2_2}}\\
  &\quad\leq C_\sigma\exp\bigg\{C_\sigma\ds{\xkh{1+\norm{\n\r}^2_{H^1}
  +\norm{\r}^2_{H^1}\norm{\n \r}^2_{H^1}+\norm{\nh \g_\e}^2_2}}\nonumber\\
  &\quad\quad+C_\sigma\ds{\zkh{(1+\e^2)\norm{V_\e}^2_2+\norm{\nh V_\e}^2_2+\norm{\p_z V_\e}^2_2
  +\norm{V_\e}^4_2+\norm{V_\e}^2_2\norm{\nh V_\e}^2_2}}\nonumber\\
  &\quad\quad+C_\sigma\ds{\zkh{\norm{v}^2_{H^1}+\norm{\n^2 v}^2_2+\norm{v}^4_2
  +(1+\e^4)\norm{v}^2_{H^1}\norm{\n v}^2_{H^1}}}\bigg\}\nonumber\\
  &\quad\quad\times\bigg\{\e^2\ds{\xkh{\norm{v}^4_2+\norm{v}^2_{H^1}\norm{\n v}^2_{H^1}
  +\norm{\n^2 v}^4_2+\norm{\n^2 v}^2_2\norm{\n\la v}^2_2+\norm{V_\e}^4_2}}\nonumber\\
  &\quad\quad+\e^2\ds{\xkh{\norm{\n \p_t v}^2_2+\norm{\n^2 v}^2_2+\norm{\n \la v}^2_2}}
  +\ds{\xkh{\norm{\n^2 v}^2_2+\norm{\n^2 \r}^2_2}\norm{V_\e}^2_2}\bigg\}\nonumber\\
  &\quad\leq C_\sigma \e^2\exp\dkh{C_\sigma(t+1)\zkh{1+\ma{M}_1(t)+\ma{M}_2(t)
  +\ma{M}_3(t)+\ma{M}^2_1(t)+\ma{M}^2_2(t)}}\nonumber\\
  &\quad\quad\times\dkh{\ma{M}_1(t)+\ma{M}_3(t)+(t+1)\zkh{\ma{M}^2_1(t)+\ma{M}^2_2(t)+\ma{M}^2_3(t)}}.
\end{flalign*}note that we have used the fact $(V_\e,W_\e,\g_\e)|_{t=0}=0$. This completes the proof.
\end{proof}

\section{Proof of Theorem~\ref{th:b12}}
With the help of Proposition \ref{po:b1T} and \ref{po:b2T}, we give the proof of Theorem \ref{th:b12}. Note that we also need to eliminate the effect of the smallness condition in Proposition \ref{po:b2T}.~This can be achieved by the following proposition.
\begin{proposition}\label{po:b12}
Denote by $T^*_\e$ the maximal existence time of the strong solution $(v_\e,w_\e,\r_\e)$ to (SBE) with (\ref{ga:are})-(\ref{eq:eve}). Then, for any $T>0$, there is a small positive constant $\e(T)=\frac{2\l_0}{3\sqrt{\ma{M}_4(T)}}$ such that $T^*_\e>T$ as long as $\e \in (0,\e(T))$. Furthermore, the following energy estimate holds
\begin{flalign*}
  &\sup_{0 \leq s \leq t}\xkh{\norm{(V_\e,\e W_\e,\g_\e)}^2_{H^1}}(s)+\ds{\xkh{\frac{1}{Re_1}\norm{\nh V_\e}^2_{H^1}+\frac{1}{Re_2}\norm{\p_z V_\e}^2_{H^1}+\frac{\e^2}{Re_1}\norm{\nh W_\e}^2_{H^1}}}\\
  &\qquad+\ds{\xkh{\frac{\e^2}{Re_2}\norm{\p_z W_\e}^2_{H^1}+\frac{1}{Rt_1}\norm{\nh \g_\e}^2_{H^1}+\frac{1}{Rt_2}\norm{\p_z \g_\e}^2_{H^1}}}\leq \e^2\xkh{\ma{M}_2(t)+\ma{M}_4(t)},
\end{flalign*}for any $t \in [0,T]$, where both $\ma{M}_2(t)$ and $\ma{M}_4(t)$ are nonnegative continuously increasing functions that do not depend on $\e$.
\end{proposition}
\begin{proof}[Proof.]
For any $T>0$, we set $\ma{T}'_\e=\min\{T^*_\e,T\}$. Then it follows from Proposition \ref{po:b1T} that
\begin{flalign}
  &\sup_{0 \leq s \leq t}\xkh{\norm{(V_\e,\e W_\e,\g_\e)}^2_2}(s)+\ds{{\xkh{\frac{1}{Re_1}\norm{\nh V_\e}^2_2+\frac{1}{Re_2}\norm{\p_z V_\e}^2_2+\frac{\e^2}{Re_1}\norm{\nh W_\e}^2_2}}}\nonumber\\
  &\qquad+\ds{\xkh{\frac{\e^2}{Re_2}\norm{\p_z W_\e}^2_2+\frac{1}{Rt_1}\norm{\nh \g_\e}^2_2+\frac{1}{Rt_2}\norm{\p_z \g_\e}^2_2}} \leq \e^2\ma{M}_2(t),\label{eq:ta1}
\end{flalign}for every $t \in [0,\ma{T}'_\e)$, where
\begin{equation*}
  \ma{M}_2(t)=Ce^{\dkh{Ct+C(t+1)\zkh{\ma{M}_1(t)+\ma{M}^2_1(t)}}}
  \zkh{\xkh{1+t}^2+\ma{M}_1(t)+\ma{M}^2_1(t)}.
\end{equation*}Define
\begin{equation*}
  t'_\e:=\sup\dkh{t \in (0,\ma{T}'_\e)\bigg|\sup_{0 \leq s \leq t}\xkh{\norm{\n (V_\e,\e W_\e,\g_\e)}^2_2}(s) \leq \l^2_0},
\end{equation*}note that $\l_0$ is the small positive constant in Proposition \ref{po:b2T}. Thanks to Proposition \ref{po:b2T}, we have the following estimate
\begin{flalign}
  &\sup_{0 \leq s \leq t}\xkh{\norm{\n(V_\e,\e W_\e,\g_\e)}^2_2}(s)+\ds{\xkh{\frac{1}{Re_1}\norm{\n\nh V_\e}^2_2
  +\frac{1}{Re_2}\norm{\n\p_z V_\e}^2_2+\frac{\e^2}{Re_1}\norm{\n\nh W_\e}^2_2}}\nonumber\\
  &\qquad+\ds{\xkh{\frac{\e^2}{Re_2}\norm{\n\p_z W_\e}^2_2+\frac{1}{Rt_1}\norm{\n\nh \g_\e}^2_2
  +\frac{1}{Rt_2}\norm{\n\p_z \g_\e}^2_2}} \leq \e^2\ma{M}_4(t),\label{eq:ta2}
\end{flalign}for every $t \in [0,t'_\e)$, where
\begin{flalign*}
  \ma{M}_4(t)&=C\exp\dkh{C(t+1)\zkh{1+\ma{M}_1(t)+\ma{M}_2(t)+\ma{M}_3(t)+\ma{M}^2_1(t)+\ma{M}^2_2(t)}}\\
  &\quad\times\dkh{\ma{M}_1(t)+\ma{M}_3(t)+(t+1)\zkh{\ma{M}^2_1(t)+\ma{M}^2_2(t)+\ma{M}^2_3(t)}}.
\end{flalign*}Let $\e(T)=\frac{2\l_0}{3\sqrt{\ma{M}_4(T)}}$. Owing to (\ref{eq:ta2}), we reach
\begin{flalign*}
  &\sup_{0 \leq s \leq t}\xkh{\norm{\n(V_\e,\e W_\e,\g_\e)}^2_2}(s)+\ds{\xkh{\frac{1}{Re_1}\norm{\n\nh V_\e}^2_2
  +\frac{1}{Re_2}\norm{\n\p_z V_\e}^2_2}}\\
  &\qquad+\ds{\xkh{\frac{\e^2}{Re_1}\norm{\n\nh W_\e}^2_2+\frac{\e^2}{Re_2}\norm{\n\p_z W_\e}^2_2
  +\frac{1}{Rt_1}\norm{\n\nh \g_\e}^2_2+\frac{1}{Rt_2}\norm{\n\p_z \g_\e}^2_2}}\\
  &\qquad\leq \frac{4\l^2_0\ma{M}_4(t)}{9\ma{M}_4(T)}\leq\frac{4\l^2_0}{9}<\l^2_0,
\end{flalign*}for every $t \in [0,t'_\e)$, and for every $\e \in (0,\e(T))$, which leads to
\begin{equation}\label{eq:l20}
  \sup_{0 \leq s < t'_\e}\xkh{\norm{\n(V_\e,\e W_\e,\g_\e)}^2_2}(s)<\l^2_0.
\end{equation}The definition of $t'_\e$ and (\ref{eq:l20}) imply that $t'_\e=\ma{T}'_\e$. On account of this, combining (\ref{eq:ta1}) with (\ref{eq:ta2}) yields
\begin{flalign}
  &\sup_{0 \leq s \leq t}\xkh{\norm{(V_\e,\e W_\e,\g_\e)}^2_{H^1}}(s)+\ds{\xkh{\frac{1}{Re_1}\norm{\nh V_\e}^2_{H^1}+\frac{1}{Re_2}\norm{\p_z V_\e}^2_{H^1}+\frac{\e^2}{Re_1}\norm{\nh W_\e}^2_{H^1}}}\nonumber\\
  &\quad+\ds{\xkh{\frac{\e^2}{Re_2}\norm{\p_z W_\e}^2_{H^1}+\frac{1}{Rt_1}\norm{\nh \g_\e}^2_{H^1}+\frac{1}{Rt_2}\norm{\p_z \g_\e}^2_{H^1}}} \leq \e^2\xkh{\ma{M}_2(t)+\ma{M}_4(t)},\label{fl:M24}
\end{flalign}for every $t \in [0,\ma{T}'_\e)$, and for every $\e \in (0,\e(T))$.

Next, it suffices to prove that $T^*_\e>T$ for every $\e \in (0,\e(T))$. Assume that $T^*_\e \leq T$, \ie, $\ma{T}'_\e=\min\{T^*_\e,T\}=T^*_\e$. Then it is clear that
\begin{equation}\label{eq:inf}
  \limsup_{t' \rightarrow (\ma{T}'_\e)^-}\xkh{\norm{(v_\e,\e w_\e,\r_\e)}^2_{H^1}}(t')=\infty,
\end{equation}since $T^*_\e$ is the maximal existence time of the strong solution $(v_\e,w_\e,\r_\e)$ to (SBE). By virtue of Proposition \ref{eq:a1t} and \ref{po:a2t}, the following estimate holds
\begin{flalign*}
  &\sup_{0 \leq t' \leq t}\xkh{\norm{(v_\e,\e w_\e,\r_\e)}^2_{H^1}}(t')\\
  &\qquad\leq\sup_{0 \leq t' \leq t}\xkh{\norm{(v,\e w,\r)}^2_{H^1}}(t')
  +\sup_{0 \leq t' \leq t}\xkh{\norm{(V_\e,\e W_\e,\g_\e)}^2_{H^1}}(t')\\
  &\qquad\leq \ma{M}_1(t)+\e^2(\ma{M}_1(t)+\ma{M}_3(t))+\e^2\xkh{\ma{M}_2(t)+\ma{M}_4(t)},
\end{flalign*}for every $t \in [0,\ma{T}'_\e)$, where the incompressible condition and estimate (\ref{fl:M24}) are used. It is obvious that (\ref{eq:inf}) contradicts to the above estimate. This contradiction deduces that $T^*_\e>T$ and hence $\ma{T}'_\e=T$, completing the proof of Proposition \ref{po:b12}.
\end{proof}

Proof of Theorem~\ref{th:b12} is shown as follows.
\begin{proof}[Proof of Theorem~\ref{th:b12}.]
For any $T>0$, from Proposition \ref{po:b12} it follows that (SBE) corresponding to (\ref{ga:are})-(\ref{eq:eve}) exists a unique strong solution $(v_\e,w_\e,\r_\e)$ on the time interval $[0,T]$ for every $\e$ $\in $ $(0,\e(T))$. Moreover, the following estimate holds
\begin{flalign*}
  &\sup_{0 \leq t \leq T}\xkh{\norm{(V_\e,\e W_\e,\g_\e)}^2_{H^1}}(t)+\int^T_0{\xkh{\frac{1}{Re_1}\norm{\nh V_\e}^2_{H^1}+\frac{1}{Re_2}\norm{\p_z V_\e}^2_{H^1}}}dt\\
  &\qquad+\int^T_0{\xkh{\frac{\e^2}{Re_1}\norm{\nh W_\e}^2_{H^1}+\frac{\e^2}{Re_2}\norm{\p_z W_\e}^2_{H^1}
  +\frac{1}{Rt_1}\norm{\nh \g_\e}^2_{H^1}+\frac{1}{Rt_2}\norm{\p_z \g_\e}^2_{H^1}}}dt\\
  &\qquad\leq \e^2\xkh{\ma{M}_2(T)+\ma{M}_4(T)}=:\e^2\ma{M}_5(T),
\end{flalign*}where $\ma{M}_5(t)$ is a nonnegative continuously increasing function that does not depend on $\e$. By virtue of the above estimate, we obtain the following strong convergences
\begin{gather*}
  (v_\e,\e w_\e,\r_\e) \rightarrow (v,0,\r)~in~L^{\infty}\xkh{0,T;H^1(\O)},\\
  \xkh{\frac{\nh v_\e}{\sqrt{Re_1}},\frac{\p_z v_\e}{\sqrt{Re_2}},\frac{\e\nh w_\e}{\sqrt{Re_1}}}
  \rightarrow \xkh{\frac{\nh v}{\sqrt{Re_1}},\frac{\p_z v}{\sqrt{Re_2}},0}~in~L^2\xkh{0,T;H^1(\O)},\\
  \xkh{\frac{\e\p_z w_\e}{\sqrt{Re_2}},\frac{\nh \r_\e}{\sqrt{Rt_1}},\frac{\p_z \r_\e}{\sqrt{Rt_2}}}
  \rightarrow \xkh{0,\frac{\nh \r}{\sqrt{Rt_1}},\frac{\p_z \r}{\sqrt{Rt_2}}}~in~L^2\xkh{0,T;H^1(\O)}.
\end{gather*}Due to the incompressible condition, it deduces from $\frac{\nh v_\e}{\sqrt{Re_1}} \rightarrow \frac{\nh v}{\sqrt{Re_1}}~in~L^2\xkh{0,T;H^1(\O)}$ and $v_\e \rightarrow v~in~L^{\infty}\xkh{0,T;H^1(\O)}$ that
\begin{gather*}
  w_\e \rightarrow w~in~L^2\xkh{0,T;H^1(\O)},\\
  w_\e \rightarrow w~in~L^{\infty}\xkh{0,T;L^2(\O)},
\end{gather*}respectively.~In addition, it is obvious that the convergence rate is of the order $O(\e)$. Theorem \ref{th:b12} is thus proved.
\end{proof}

\noindent{\bf Acknowledgments.~}\small
The work of X. Pu was supported in part by the National Natural Science Foundation of China (No. 11871172) and the Natural Science Foundation of Guangdong Province of China (No. 2019A1515012000). The work of W. Zhou was supported by the Innovation Research for the Postgraduates of Guangzhou University (No. 2021GDJC-D09).


\end{document}